\documentclass[12pt]{amsart}

\usepackage{amsmath,amssymb,amsthm}
\usepackage{amsfonts}
\usepackage[mathscr]{eucal}
\newtheorem{theorem}{Theorem}[section]

\theoremstyle{definition}

\theoremstyle{remark}

\numberwithin{equation}{section}

\newcommand{\ep}{\varepsilon}
\newcommand{\rr}{{\mathbb R}}

\def\R{{\mathbb R}}

\def\x{{\bf x}}

\def\ep{\varepsilon}

\def\P{{\mathbb P}}

\newcommand{\Exp}{\mathbb E}
\newcommand{\Var}{\operatorname{Var}}

\begin{document}

\title{\bf Correlated continuous time random walks}

\author{Mark M. Meerschaert}
\address{Mark M. Meerschaert, Department of Probability and Statistics,
Michigan State University, East Lansing, MI 48823}
\email{mcubed@stt.msu.edu}
\urladdr{http://www.stt.msu.edu/$\sim$mcubed/}
\thanks{Research of M. M. Meerschart was partially
supported by NSF grant DMS-0706440.}

\author{Erkan Nane}
\address{Erkan Nane, Department of Mathematics and Statistics, 221
Parker Hall, Auburn University, Auburn, AL 36849}
\email{nane@stt.msu.edu}
\urladdr{http://www.stt.msu.edu/$\sim$nane}

\author{Yimin Xiao}
\address{Yimin Xiao, Department Statistics and Probability,
Michigan State University, East Lansing, MI 48823}
\email{xiao@stt.msu.edu}
\urladdr{http://www.stt.msu.edu/$\sim$xiaoyimi}
\thanks{Research of Y. Xiao was partially supported by
NSF grant DMS-0706728.}

\begin{abstract}
Continuous time random walks impose a random waiting time before each particle jump.
Scaling limits of heavy tailed continuous time random walks are governed by fractional
evolution equations.  Space-fractional derivatives describe heavy tailed jumps, and
the time-fractional version codes heavy tailed waiting times.  This paper develops
scaling limits and governing equations in the case of correlated jumps.  For long-range
dependent jumps, this leads to fractional Brownian motion or linear fractional stable
motion, with the time parameter replaced by an inverse stable subordinator in the case
of heavy tailed waiting times. These scaling limits provide an interesting class of
non-Markovian, non-Gaussian self-similar processes.
\end{abstract}


\maketitle
\section{Introduction}
Continuous time random walks (CTRW) separate IID particle jumps $(Y_n)$ by IID waiting
times $(J_n)$.  CTRW models are important in applications to geology \cite{Berkowitz2006},
physics \cite{MetzlerKlafter}, and finance \cite{scalas1}. In the case of heavy tailed
waiting times, CTRW scaling limits are subordinated processes that are self-similar but
non-Markovian \cite{limitCTRW}.  Their transition densities are governed by fractional
diffusion equations \cite{Zsolution}.  Fractional diffusion equations replace the usual
integer order derivatives in the diffusion equation by their fractional-order analogues
\cite{MillerRoss, Samko}.  Just as the diffusion equation $\partial_t u=a\partial_x^2 u$
governs the scaling limit of a simple random walk, the fractional diffusion equation
$\partial_t^\beta u=a\partial_x^\alpha u$ governs the scaling limit of a CTRW with heavy
tail jumps $\P(Y_n>r)\sim r^{-\alpha}$ for $0<\alpha<2$ and waiting times $\P(J_n>t)\sim
t^{-\beta}$ for $0<\beta<1$.

This paper develops limit theorems and governing equations for CTRW with correlated jumps
$Y_n=\sum_j c_j Z_{n-j}$, where $(Z_n)$ are IID and  $(c_n)$ are real numbers (see Section
2 for precise conditions). These CTRW
models are useful for correlated observations separated by random waiting times, which
are common, for example, in finance \cite{SGM}. Scaling limits of the partial sum process
$S(t)=Y_1+\cdots+Y_{[t]}$ in the case of long range dependence include fractional Brownian
motion (FBM) for light-tailed jumps \cite{davydov,whitt}, and linear fractional stable
motion (LFSM) for heavy-tailed jumps \cite{Astrauskus,kasahara-maejima,whitt}. Letting
$T_n =J_1+\cdots+J_n$ the time of the $n$th jump,
and $N_t=\max\{n:T_n\leq t\}$ the number of jumps by time $t>0$, the scaling limit
of the CTRW $S(N_t)$ is a FBM or LFSM subordinated to an inverse stable subordinator,
which is connected to the local time of a strictly stable L\'evy process \cite{MNX},
or the supremum process of a spectrally negative stable L\'evy process \cite{Bingham73}.
This extends the results of \cite{limitCTRW,CTRW} to the case of dependent jumps.
We also discuss some interesting properties of these self-similar limit processes,
and governing equations for their probability densities.

\bigskip

\section{Results}


Let $\{Z_n, -\infty < n < \infty\}$ denote IID random variables that
belong to the strict domain of attraction of some stable law $A$
with index $0<\alpha\leq 2$. This
means that the sequence of partial sums $P(n)=Z_1+\cdots+Z_n$ satisfies
$a_nP(n)\Rightarrow A$ for some $a_n>0$, see Feller \cite[p.312--313]{Feller}
or Whitt \cite[p.114--115]{whitt}.  

The particle jumps that we consider in this paper are given by the stationary
linear process $\{Y_n, -\infty < n < \infty\}$ defined by
$Y_n= \sum_{j=0}^\infty c_j Z_{n-j}$, where $c_j$ are real constants
such that $\sum_{j=0}^\infty |c_j|^\rho < \infty$ for some $\rho \in (0, \alpha)$.
This condition ensures that the the series $\sum_{j=0}^\infty c_j Z_{n-j}$
converges in $L^\rho(\P)$ and almost surely (see Avram and Taqqu \cite{AvramTaqqu92}).

Let $J_n>0$ be IID waiting times, $T_n=J_1+\cdots+J_n$ the time of the $n$th
particle jump, and $N_t=\max\{n:T_n\leq t\}$ the
number of jumps by time $t>0$.  Let $S(n)=Y_1+\cdots+Y_n$ denote the location
of the particle after $n$ jumps, so that the continuous time random walk (CTRW)
$S(N_t)$ gives the location of the particle at time $t>0$.  Suppose that
$J_n$ belongs to the domain of attraction of
some stable law $D$ with index $0<\beta< 1$ and $D > 0$ almost surely. Hence
$b_nT_n\Rightarrow D$ for some norming
constants $b_n>0$.  Let $b(t)=b_{[t]}$ and take $\tilde b(t)$ an asymptotic inverse of
the regularly varying function $1/b(t)$, so that $tb(\tilde b(t))\to 1$ as $t\to\infty$
\cite{limitCTRW}.

Let $\{A(t), t\ge 0\}$ and $\{D(t), t \ge 0\}$ be stable L\'evy processes with $A(1) =A,
D(1) = D$, respectively. Note that $\{D(t), t \ge 0\}$ is a stable subordinator
of index $\beta$, hence its sample functions are almost surely
strictly increasing \cite[p.75]{Bertoin96}. Therefore, the inverse or hitting time process
of $\{D(t), t \ge 0\}$,
$$
E_t=\inf\{x>0:D(x)>t\}, \qquad \forall t \ge 0
$$
is well defined and the function $t \mapsto E_t$ is strictly increasing almost surely.

Our first result shows that the CTRW scaling limit in the case of short-range dependence
is quite similar to the case of independent jumps studied by Meerschaert and
Scheffler \cite{limitCTRW}.

\begin{theorem}\label{SRDtheorem}
Under the conditions of this section, suppose that $0 < \alpha<2$, $c_j\geq 0$
and $\sum_j c_j^\rho<\infty$ for some $\rho\in(0,\alpha)$ with $\rho\leq 1$,
and that one of the following holds:
\begin{enumerate}
\item[(a)] $0<\alpha\leq 1$; or
\item[(b)] $c_j=0$ for all but finitely many $j$; or
\item[(c)] $1< \alpha <2$, $c_j$ is monotone and $\sum_j c_j^\rho<\infty$
for some $\rho<1$.
\end{enumerate}
Then we have
\begin{equation}\label{SRDconv}
w^{-1} a_{[\tilde b(c)]} S(N_{ct})\Rightarrow A(E_t)
\end{equation}
as $c\to\infty$ in the $M_1$ topology on $D([0,\infty),\R)$, where $w=\sum_{j} c_j$.
\end{theorem}

In view of Avram and Taqqu \cite{AvramTaqqu92}, the convergence in \eqref{SRDconv}
cannot be strengthened to the $J_1$ topology. Note that the processes $\{A(t),
t\ge 0\}$ and $\{E_t, t \ge 0\}$ are self-similar, that is, for every constant
$c > 0$
$$ \{A(ct), t \ge 0\} \stackrel{d} {=}
\{ c^{1/\alpha} A(t), t \ge 0\}$$
and
$$ \{E_{ct}, t \ge 0\} \stackrel{d} {=} \{c^\beta E_t, t \ge 0\},$$
where $\stackrel{d} {=}$ means equality in all finite dimensional distributions.
It follows immediately that the scaling limit $\{A(E_t), t \ge 0\}$ in \eqref{SRDconv}
is self-similar with index $\beta/\alpha$.
When $0<\beta\leq 1/2$, the inner process $E_t$ in \eqref{SRDconv} is also the
local time at zero of a strictly stable L\'evy motion \cite{MNX}.  When
$1/2\leq \beta<1$, the inner process $E_t$ is also the supremum process of
a stable L\'evy motion with index $1/\beta$ and no negative jumps \cite{Bingham73}.

Let $\partial_t^\beta g(t)$ denote the Caputo fractional derivative, the inverse Laplace
transform of $s^\beta \tilde g(s)-s^{\beta-1}g(0)$ where $\tilde g(s)=
\int_0^\infty e^{-st}\,g(t)\,dt$ is the usual Laplace transform of $g$.  Let
$\partial^\alpha_{\pm x} f(x)$ denote the Liouville fractional derivative, the
inverse Fourier transform of $(\pm ik)^\alpha\hat f(k)$, where $\hat f(k)=
\int_{-\infty}^\infty e^{-ikx}f(x)\,ds$ is the usual Fourier transform.
The stable random variable $A(t)$ has a smooth density with Fourier transform
$e^{-t\psi(k)}$ where $\psi(k)=a[p(ik)^\alpha+q(-ik)^\alpha]$ with $0\leq p,q\leq 1$
and $p+q=1$ \cite{RVbook}.  Then the limit $A(E_t)$ in \eqref{SRDconv} has a
density $h(x,t)$ that solves the fractional diffusion
equation $\partial_t^\beta h =a p\partial_x^\alpha h + a q \partial_{-x}^\alpha h$,
see \cite{limitCTRW}.

Next we consider the CTRW scaling limit for heavy-tailed particle jumps with
long-range dependence. To simplify the presentation, we assume $a_n=n^{-1/\alpha}$
(domain of normal attraction) and power-law weights; namely $c_j\sim
c_0 j^{H-1-1/\alpha}$ as $j \to \infty$, for some $c_0>0$. The case $0 < H < 1/\alpha$
means the stationary sequence $\{Y_n\}$ has short-range dependence, while the
case $1/\alpha<H<1$ means $\{Y_n\}$ has long-range dependence. The scaling
limit of CTRW with short-range dependence has been partially covered by Theorem
\ref{SRDtheorem}. The rest of the cases are treated in Theorems \ref{LRDthm2} and
\ref{LRDthm3} below.

We will make use of the following definition. Given constants $\alpha \in (0, 2)$
and $H \in (0, 1)$, the $\alpha$-stable process $\{L_{\alpha,H}(t), t \in \R\}$
defined by
\begin{equation}\label{Def:LFSS}
L_{\alpha,H}(t)=\int_{\R}\left[(t-s)_+^{H-1/\alpha}-(-s)_+^{H-1/\alpha}\right]\,A(ds)
\end{equation}
is called a linear fractional stable motion (LFSM) with indices $\alpha$ and $H$.
In the above, $a_+ = \max\{0, a\}$ for all $ a \in \R$ and $\{A(t), t \in \R\}$ is
a two-sided strictly stable L\'evy process
of index $\alpha$ with $A(1) = A$ given at the beginning of Section 2. Namely,
$n^{-1/\alpha} P(n) \Rightarrow A$ as $n \to \infty$.
Because of this, $\{L_{\alpha,H}(t), t \in \R\}$ defined by (2.2) differs
from the LFSM in Theorem 4.7.2 in \cite{whitt} by a constant factor.
Note that, when $H =  1/\alpha$, $L_{\alpha,H}(t) = A(t)$ for
all $t \ge 0$. When  $H \ne 1/\alpha$, the stochastic integral in (\ref{Def:LFSS})
is well-defined because
\[
\int_{\R}\Big|(t-s)_+^{H-1/\alpha}-(-s)_+^{H-1/\alpha}\Big|^\alpha\,dr
< \infty.
\]
See \cite[Chapter 3]{ST94}.

By (\ref{Def:LFSS}), it can be verified that $\{L_{\alpha,H}(t), t \in \R\}$ is
$H$-self-similar with stationary increments \cite[Proposition 7.4.2]{ST94}. It
is an $\alpha$-stable analogue of fractional Brownian motion and its
probabilistic and statistical properties have been investigated by many authors.
In particular, it is known that
\begin{itemize}
\item[(i)]\ If $1/\alpha<H<1$ (this is possible only when $1 < \alpha < 2$),
then the sample
function of $\{L_{\alpha,H}(t), t \in \R\}$ is almost surely continuous.
\item[(ii)]\ If $ 0< H<1/\alpha$, then the the sample function of $\{L_{\alpha,H}(t),
t \in \R\}$ is almost surely unbounded on every interval of positive length.
\end{itemize}
We refer to \cite[Chapters 10 and 12]{ST94} for more information.

\begin{theorem}\label{LRDthm2}
We assume the setting of this section.  If $1<\alpha<2$, $1/\alpha<H<1$,
and $c_j \sim c_0 j^{H-1-1/\alpha}$ as $j \to \infty$ for some $c_0>0$,
then as $c \to \infty$
\begin{equation}\label{LRDconv2}
[\tilde b(c)]^{-H} S(N_{ct})\Rightarrow K_1\, L_{\alpha,H}(E_t)
\end{equation}
in the $J_1$ topology on $D([0,\infty),\R)$, where $K_1 =
c_0\alpha/(H \alpha -1)$.
\end{theorem}

The topology on $D([0,\infty),\R)$ in Theorem \ref{LRDthm2} is stronger than that
in Theorem \ref{SRDtheorem}, thanks to the fact that $L_{\alpha,H}(t)$ is a.s.
continuous whenever $1/\alpha<H<1$.

Observe that the case when $0 < H < 1/\alpha$ and the constants $c_j$ ($j \ge 0$)
are not all nonnegative is left uncovered by Theorems \ref{SRDtheorem}
and \ref{LRDthm2}. Because of Property (ii) of $\{L_{\alpha,H}(t), t \in \R\}$,
the limiting process does not belong to the function space $D([0,\infty),\R)$.
Nevertheless, we have the following theorem.

\begin{theorem}\label{LRDthm3}
We assume the setting of this section. If $0<\alpha <2$, $0 < H < 1/\alpha$, $c_j\sim
c_0 j^{H-1-1/\alpha}$ as $j \to \infty$ for some $c_0>0$, and $\sum_j c_j = 0$,
then
\begin{equation}\label{LRDconv3}
[\tilde b(c)]^{-H} S(N_{ct}) \stackrel{f.d.} {\longrightarrow} K_1\, L_{\alpha,H}(E_t)
\end{equation}
as $c \to \infty$, where  $\stackrel{f.d.} {\longrightarrow}$ means convergence of all
finite-dimensional distributions and $K_1 =
c_0\alpha/(H \alpha -1)$.
\end{theorem}

It is interesting to note that the constants in  Theorems \ref{LRDthm2}
and \ref{LRDthm3} are determined by $c_0$, $\alpha$ and $H$ in the same way.
But $K_1$ is positive is when $1/\alpha<H<1$, and is
negative when $0 < H < 1/\alpha$.

It follows from the self-similarity of $\{L_{\alpha,H}(t), t \in \R\}$ and $\{E(t),
t \ge 0\}$ that the scaling limits in \eqref{LRDconv2} and \eqref{LRDconv3} are
self-similar with index $H\beta$. When $1/\alpha<H<1$, it can be seen that
$\{L_{\alpha,H}(E_t), t \ge 0\}$ has continuous sample functions almost
surely. However, if $0 < H < 1/\alpha$, then $\{L_{\alpha,H}(E_t), t \ge 0\}$ is
almost surely unbounded on every interval of positive length. It would be interesting
to further study the properties of the process $\{L_{\alpha,H}(E_t), t \ge 0\}$.

We mention that both Theorems \ref{LRDthm2} and \ref{LRDthm3} can be extended to
$(Z_n)$ in the strict domain of attraction of
$A$ and $(c_j)$ regularly varying at $\infty$ with index $H-1-1/\alpha$,
using a slightly different normalization in \eqref{LRDconv2} depending on
$a_n$ and the probability tail of $Z_n$, compare \cite{Astrauskus}.

Finally we consider the case $\alpha = 2$. If $\{A(t), t \in \R\}$ in (\ref{Def:LFSS})
is replaced by ordinary two-sided Brownian motion, then  (\ref{Def:LFSS}) defines
a fractional Brownian motion $W_H = \{W_H(t), t \in \R\}$ on $\R$ of index $H$,
which is a Gaussian process with mean zero and covariance function
$$\Exp[W_H(t)W_H(s)]= \frac 1 2 \big[|t|^{2H}+|s|^{2H}-|t-s|^{2H}\big]. $$

Theorem \ref{LRDthm} gives the CTRW scaling limit for light-tailed particle jumps
with long-range dependence.

\begin{theorem}\label{LRDthm}
We assume the setting of this section. If $\alpha=2$, $\Exp[Z_n]=0$, $\Exp[Z_n^2]<\infty$,
$\sum_j c_j^2<\infty$, $\Var[S(n)]=\sigma_n^2$ varies regularly at $\infty$ with
index $2H$ for some $0<H<1$, and $\Exp[S(n)^{2\rho}]\leq K_2\,
\big[\Exp (S(n)^2)\big]^\rho$ for some constants $K_2>0$ and $\rho>1/H$, then
as $c \to \infty$
\begin{equation}\label{LRDconv}
\sigma^{-1}_{[\tilde b(c)]} S(N_{ct})\Rightarrow W_H(E_t)
\end{equation}
in the $J_1$ topology on $D([0,\infty),\R)$.
\end{theorem}

Note that it is not difficult to provide examples of sequences of IID random
variables $\{Z_n\}$ and real numbers $\{c_j\}$ that satisfy the conditions
of Theorem \ref{LRDthm}, see \cite{davydov, Giriatis03}. It follows from the results of
Taqqu \cite{taqqu} that the conclusion of Theorem \ref{LRDthm}
still holds if the linear process $\{Y_n\}$ is replaced by the stationary sequence
$\{g(\xi_n)\}$, where $\{\xi_n\}$ is a stationary Gaussian sequence with mean 0,
variance 1 and long-range dependence, and $g \in L^2(e^{-x^2/2}dx)$
is a function with Hermite rank $1$.

Theorem \ref{LRDthm} contains the case $H=1/2$ where $W_H(t)=A(t)$ is a
standard Brownian motion. This includes the situation of mean zero finite
variance particle jumps, and heavy tailed waiting times between jumps.
In this case, the CTRW scaling limit $A(E_t)$ has a density $h(x,t)$ that
solves the time-fractional diffusion equation $\partial_t^\beta h=a\,\partial_x^2 h$,
see \cite{limitCTRW}.  Since $\{W_H(ct), t \ge 0\} \stackrel{d}{=}
\{c^{H} W_H(t), t \ge 0\}$, the scaling limit in
\eqref{LRDconv} is self-similar with index $H\beta$. Some results on large
deviation and sample path regularity have recently been obtained for
$\{W_H(E_t), t \ge 0\}$ in \cite{MNX}.

In the case of finite mean waiting times, the CTRW scaling limit is essentially
the same as for the underlying random walk.  If $\mu=\Exp J_n<\infty$, then
$\mu N_t/t\to 1$ almost surely as $t\to\infty$, and a simple argument along
the lines of the proof of Theorem \ref{SRDtheorem} shows that $w^{-1} a_{[c]}
S(N_{ct})\Rightarrow A(t/\mu)$ in the $M_1$ topology on $D([0,\infty),\R)$.
Theorems \ref{LRDthm2}, \ref{LRDthm3} and \ref{LRDthm} can be extended similarly.

An easy argument with Fourier transforms shows that the density $h(x,t)$ of
$L_{\alpha,H}(t)$ solves $\partial_t h=\alpha
Ht^{\alpha H-1}[ap\partial_x^\alpha h + aq \partial_{-x}^\alpha h]$.
A similar argument shows that the density of $W_H(t)$ solves
$\partial_t h=2Ht^{2H-1}a\partial_x^2 h$. An interesting open
question is to establish the governing equation for the CTRW 
scaling limit in \eqref{LRDconv2} and \eqref{LRDconv}. This is 
not as simple as replacing the first time derivative by a 
fractional derivative in the governing equation for
the outer process, since the $t$ variable also appears on 
the right-hand side, so that Theorem 3.1 of \cite{fracCauchy} 
does not apply.

\bigskip
\section{Proofs}

The proofs in this section are based on invariance principles for 
stationary sequences with short or long-range dependence
(see, for example, \cite{whitt}) and the CTRW 
limit theory developed in \cite{limitCTRW}. Due to the non-Markovian 
nature of the CTRW scaling limits in this paper, standard subordination 
methods can not be applied directly. Instead we apply continuous 
mapping-type arguments to prove Theorems \ref{SRDtheorem}, 
\ref{LRDthm2} and \ref{LRDthm}. The proof of Theorem \ref{LRDthm3} 
is quite different and relies on a criterion for the convergence 
of all finite-dimensional distributions of composite processes established 
by Becker-Kern, Meerschaert and Scheffler \cite{CTRW}.

Recall that $J_n>0$ are IID waiting times, $T_n=J_1+\cdots+J_n$ the time of
the $n$th particle jump, and $N_t=\max\{n:T_n\leq t\}$ the number of jumps by
time $t>0$.  Since $J_n$ belongs to the domain of attraction of some stable
law $D$ with index $0<\beta< 1$ and $D > 0$ almost surely, with
$b_nT_n\Rightarrow D$ for some norming
constants $b_n>0$, the sequence $b_n$ varies regularly with index $-1/\beta$,
see \cite{Feller}.  Then the asymptotic inverse $\tilde b(t)$ of $1/b$ varies
regularly with index $\beta$, see Seneta \cite{seneta}.  Recall that the stable
L\'evy motion $\{D(x), x \ge 0\}$ with $D(1)=D$ is a stable subordinator of
index $\beta$ and thus is strictly increasing \cite[p.75]{Bertoin96}.
Its inverse or hitting time process $E_t=\inf\{x>0:D(x)>t\}$ is strictly increasing
with continuous sample paths, has moments of all orders, and its increments are
neither stationary nor independent \cite{limitCTRW}. Bingham \cite{bingham}
shows that $E_t$ has a Mittag-Leffler distribution, and gives a differential
equation that governs its finite dimensional distributions.

\begin{proof}[Proof of Theorem \ref{SRDtheorem}]
Corollary 3.4 in \cite{limitCTRW} shows that $\tilde b(c)^{-1}N_{ct} \Rightarrow E_t $
as $c\to\infty$ in the $J_1$ topology on $D([0,\infty),[0,\infty))$.  Note that
$\tilde b(c)\to\infty$ as $c\to\infty$ since this function is regularly varying at
$\infty$ with index $\beta>0$.  Theorem 4.7.1 in Whitt \cite{whitt} shows that $a_n S(nt)
\Rightarrow w A(t)$ in the $M_1$ topology on $D([0,\infty),\rr)$.  Since the $J_1$
topology is stronger, and since the waiting times $(J_n)$ are independent of
$(Y_n)$, we have
$$
\big(a_{[\tilde b(c)]} S(\tilde b(c)t),\tilde b(c)^{-1}N_{ct}\big)\, \Rightarrow\,
(A(t),E_t)
$$
in the $M_1$ topology of the product space $D([0,\infty),\rr\times [0,\infty))$.
Of course, this last statement also follows from Theorem 3.2 in \cite{billingsley}.

Since the process $\{E_t, t \ge 0\}$ is strictly increasing
and continuous,  Theorem 13.2.4 in \cite{whitt} yields
\[
a_{[\tilde b(c)]}S\big(\tilde b(c)\cdot\tilde b(c)^{-1}N_{ct}\big)\, \Rightarrow
\, A(E(t))
\]
in the $M_1$ topology on $D([0,\infty),\rr)$, which completes the proof.
\end{proof}

\begin{proof}[Proof of Theorem \ref{LRDthm2}]
Recall that $\tilde b(c)^{-1}N_{ct} \Rightarrow E_t $ in the $J_1$ topology
on $D([0,\infty),$ $[0,\infty))$ \cite[Corollary 3.4]{limitCTRW}.
Theorem 4.7.2 in Whitt \cite{whitt}, originally due to
Astrauskas \cite{Astrauskus}, shows that $n^{-H} S(nt)\Rightarrow K_1\,
L_{\alpha,H}(t)$ in the $J_1$ topology on
$D([0,\infty),\rr)$, where $K_1 =c_0\alpha/(H \alpha -1)$.

Since $\{N_t, t \ge 0\}$ is independent of $\{S(n), n
\ge 1\}$, we have
$$([\tilde b(c)]^{-H} S(\tilde b(c)t),\tilde
b(c)^{-1}N_{ct}) \Rightarrow (K_1\, L_{\alpha,H}(t),E_t)$$
in the product space.
Combining this with \cite[Theorem 13.2.4]{whitt} yields \eqref{LRDconv}
in the $M_1$ topology. Since both processes $\{L_{\alpha,H}(t), t \ge 0\}$
and $\{E_t, t \ge 0\}$  are continuous, and the latter is strictly increasing,
one can apply Theorem 13.3.1 in \cite{whitt} to strengthen to the conclusion
to convergence in the $J_1$ topology. This proves Theorem \ref{LRDthm2}.
\end{proof}

\begin{proof}[Proof of Theorem \ref{LRDthm3}]
It is sufficient to show that for all integers $m \ge 1$, $0 < t_1 < \cdots < t_m$,
we have
\begin{equation}\label{Eq:Con-fidi}
\tilde b(c)^{-H}\big( S(N_{ct_1}), \ldots, S(N_{ct_m})\big) \Rightarrow
\, K_1\big( L_{\alpha,H}(E_{t_1}), \ldots, L_{\alpha,H}(E_{t_m})\big)
\end{equation}
as $c \to \infty$. For this purpose, we will make use of Proposition 4.1 in
\cite{CTRW}, which provides a useful criterion for the convergence of all
finite-dimensional distributions of composite processes,
and Corollary 3.3 in \cite{kasahara-maejima} which is concerned with convergence
of finite dimensional distributions of weighted partial sums of IID
random variables.

We will adopt some notation from \cite{CTRW}. For ${\bf t} = (t_1,
\ldots, t_m)$ and $c > 0$, let $\rho_c :=  \rho_c^{{\bf t}}$ be the distribution
of $\tilde b(c)^{-1} \,(N_{ct_1}, \ldots, N_{ct_m})$, and let
$\rho := \rho^{{\bf t}}$ be the distribution of $(E_{t_1}, \ldots, E_{t_m})$.
Since $\tilde b(c)^{-1} N_{ct} \Rightarrow E_t $ in the $J_1$ topology
on $D([0,\infty),$ $[0,\infty))$ \cite[Corollary 3.4]{limitCTRW}, we have
$\rho_c \Rightarrow \rho$ as $c \to \infty$.

It follows from the definition of $\{Y_n\}$ that, for every $x \ge 0$,
$S(nx)$ can be rewritten as
\begin{equation}\label{Eq:Sn}
S(nx) = \sum_{j=- \infty}^{\infty} \bigg(\sum_{k=1-j}^{[nx]-j}
\widetilde{c}_k\bigg) Z_j,
\end{equation}
where $\widetilde{c}_k= c_k$ if $k \ge 0$ and $\widetilde{c}_k = 0$
if $k < 0$. Under the assumptions of Theorem  \ref{LRDthm3}, we have
$\sum_{k=-\infty}^{\infty} |\widetilde{c}_k|
< \infty$, $\sum_{k=-\infty}^{\infty} \widetilde{c}_k = 0$ and
\[
\sum_{k=n}^\infty \widetilde{c}_k \sim c_0 \sum_{k=n}^\infty
k^{H- 1/\alpha - 1} \sim - \frac{c_0 \alpha} {H \alpha-1}\,
n^{H - 1/\alpha}
\]
as $n \to \infty$. Thus, the conditions of Theorem 5.2 in
\cite{kasahara-maejima} are satisfied with $\psi(n) = n^{H- 1/\alpha}$,
$a = -K_1$ (recall that $K_1 =c_0 \alpha/(H \alpha-1)$),
$b=0$ and $A=0$. It follows that
\begin{equation}\label{Eq:KM88}
n^{-H}\, S(nt) \stackrel{f.d.}{\longrightarrow}\,
K_1\, L_{\alpha,H}(t)
\quad \hbox{ as }\ n \to \infty.
\end{equation}


For any $\x = (x_1, \ldots, x_m) \in \R^m_+$, let $\mu_c(\x)$ be the
distribution of $\tilde b(c)^{-H}\big( S(\tilde b(c) x_1),$ $\ldots,
S(\tilde b(c) x_m)\big)$ and let $\nu(\x)$ be
the distribution of $K_1\big( L_{\alpha,H}(x_1), \ldots, L_{\alpha,H}(x_m)\big)$.
Then for every $c> 0$,  the mapping $\x \mapsto \mu_c(\x)$ is weakly measurable.
Since the linear fractional stable motion $\{L_{\alpha,H}(t), t \ge 0\}$
is stochastically continuous, the mapping $\x \mapsto \nu(\x)$ is weakly continuous.
Moreover, it follows from (\ref{Eq:KM88}) that, for every $\x \in \R^m_+$,
$\mu_c(\x) \Rightarrow \nu(\x)$ as $c \to \infty$.

As in \cite{CTRW} we apply a conditioning argument and the independence
between the sequences $\{Y_n\}$ and $\{J_n\}$ to derive that the
distribution of  $\tilde b(c)^{-H}
\big(S(N_{ct_1}), \ldots, S(N_{ct_m})\big)$ can be written as $\int_{\R^m_+}
\mu_c(\x)\, d\rho_c(\x)$, which is a probability measure on $\R^m$. Similarly,
the distribution of the random vector $K_1\big( L_{\alpha,H}(E_{t_1}),
\ldots, L_{\alpha,H}(E_{t_m})\big)$
can be written as $\int_{\R^m_+}\nu(\x)\, d\rho(\x)$.

Therefore, (\ref{Eq:Con-fidi}) follows from Proposition 4.1 in \cite{CTRW}
once we verify that, for every $\x \in \R^m_+$, $\mu_c(\x^{(c)}) \Rightarrow \nu(\x)$
for every sequence $\{\x^{(c)}\} \subset  \R^m_+$ that satisfies $\x^{(c)} \to \x$
as $c \to \infty$.

The last statement is equivalent to
\begin{equation}\label{Eq:KM1}
c^{-H}\,\big(S(cx_1^{(c)}), \ldots, S(cx_m^{(c)})\big)\  {\Rightarrow}\
K_1\, \big(L_{\alpha,H}(x_1), \ldots, L_{\alpha,H}(x_m)\big)
\end{equation}
whenever $\x^{(c)} \to \x$ as $c \to \infty$. This is stronger than
(\ref{Eq:KM88}), where the fixed time-instants $0< x_1
< x_2< \cdots < x_m$ are replaced now by $x_1^{(c)}, \ldots, x^{(c)}_m$.
Our proof of (\ref{Eq:KM1}) is a modification of the proof of
Theorem 5.2 in Kasahara and Maejima \cite{kasahara-maejima}.

To this end, we define the step function $ r\mapsto A_c(r)$ on $\R$ by
\begin{equation}\label{Eq:Ac}
 A_c(r) = \left\{ \begin{array}{ll}
 c^{-1/\alpha}\sum_{j=1}^{[cr]} Z_j \quad &\hbox{ if }\, r > 0,\\
 c^{-1/\alpha} \sum_{j=[cr]}^{0} Z_j &\hbox{ if }\, r \le 0.
\end{array}
\right.
\end{equation}
Then for any function $g$ on $\R$, as in \cite[p.88]{kasahara-maejima},
we define
\begin{equation}\label{Eq:Ac2}
\int_{- \infty}^\infty g(r)\, dA_c(r) = \frac 1 {c^{1/\alpha}}\,
 \sum_{j=-\infty}^\infty g\big(\frac {j} {c}\big)\, Z_j.
\end{equation}

By using (\ref{Eq:Sn}), (\ref{Eq:Ac}) and (\ref{Eq:Ac2}) we can rewrite
$c^{-H}S(cx)$ ($x > 0$ and $c > 0$) as
\begin{equation}\label{Eq:Sn2}
\begin{split}
c^{-H}S(cx) &=  \frac 1 {c^{1/\alpha}}\, \sum_{j=-\infty}^\infty
\frac1 {c^{H - 1/\alpha}} \bigg(\sum_{k=1-j}^{[cx]-j}
\widetilde{c}_k\bigg) Z_j \\
&= \frac 1 {c^{1/\alpha}} \int_{\R} g_c(x, r) dA_c(r),
\end{split}
\end{equation}
where the integrand $g_c(x, r)$ is given by
\begin{equation}\label{Eq:g}
\begin{split}
g_c(x, r) &= \frac1 {c^{H - 1/\alpha}} \,\sum_{k=1-[cr]}^{[cx]-[cr]}
\widetilde{c}_k\\
&= \frac1 {c^{H - 1/\alpha}} \bigg(\sum_{k= -[cr]+1}^\infty \widetilde{c}_k -
\sum_{k= [cx]-[cr]+1}^\infty \widetilde{c}_k\bigg).
\end{split}
\end{equation}
In the above, we have used  the fact that
$\sum_{j=-\infty}^\infty \widetilde{c}_j = 0$ to derive the second equality.

It follows from (\ref{Eq:Sn2}) that (\ref{Eq:KM1}) can be rewritten as
\begin{equation}\label{Eq:KM3}
\bigg\{\int_{\R} g_c(x_i^{(c)}, r)\, dA_c(r)\bigg\}_{i=1}^m \ \Rightarrow\
\bigg\{K_1 \int_{\R} g(x_i, r)\, dA(r)\bigg\}_{i=1}^m,
\end{equation}
where $g(x, r)= (x-r)_+^{H - 1/\alpha} - (-r)_+^{H - 1/\alpha}$
is the function in (\ref{Def:LFSS}).

Now let us fix $\x = (x_1, \ldots, x_m) \in \R_+^m$ and an
arbitrary sequence $\{\x^{(c)}\} \subset  \R^m_+$ that satisfies
$\x^{(c)} \to \x$ as $c \to \infty$.  By Corollary 3.3 in
\cite{kasahara-maejima} (with $f_n^i(\cdot)$
being taken as $g_c(x_i^{(c)}, \cdot)$), the convergence in (\ref{Eq:KM3})
will follow once we verify that for every $1 \le i \le m$ the following
conditions are satisfied:
\begin{itemize}
\item[(A1)$'$]\, for $dr$-almost every $r \in \R$,
\begin{equation}\label{Eq:g2}
g_c(x_i^{(c)}, r_c) \longrightarrow \ K_1\, g(x_i, r)
\end{equation}
whenever $r_c \to r$ as $c \to \infty$.

\item[(A2)$'$]\, for every $T > 0$, there exists a constant
$\beta > \alpha$ such that
\begin{equation}\label{Eq:g3}
\sup_{c \ge 1} \int_{|r|\le T} \big|g_c(x_i^{(c)}, r)\big|^\beta
\,d\rho_c(r) < \infty,
\end{equation}
where $\rho_c(r) = [cr]/c$, and
\item[(A3)$'$]\, there exists an $\ep > 0$ such that
\begin{equation}\label{Eq:g4}
\lim_{T \to \infty}\limsup_{c \to \infty} \int_{|r| >T}
\Big\{\big|g_c(x_i^{(c)}, r)\big|^{\alpha - \ep}
+ \big|g_c(x_i^{(c)}, r)\big|^{\alpha + \ep} \Big\}
  \,d\rho_c(r) = 0.
\end{equation}
\end{itemize}

For simplicity of notation, we will from now on omit the subscript $i$.
To verify Condition (A1)$'$, note that by the property of $\{c_k\}$,
we have
\begin{equation}\label{Eq:g5}
\lim_{c \to \infty} \frac{1} {c^{H -1/\alpha}}\sum_{k = [cr]+1}^\infty
\widetilde{c}_k = \left\{\begin{array}{ll}
- K_1 r^{H - 1/\alpha} \ \ &\hbox{ if } r > 0,\\
0 &\hbox{ if } r \le 0,
\end{array}
\right.
\end{equation}
and the convergence is uniform in $r$ on every compact set in
$\R\backslash\{0\}$. For any $x \in \R_+$ and $r\in \R$, we may distinguish
three cases $r < 0$, $0 \le r \le x$ and  $r > x$. By applying (\ref{Eq:g5})
to (\ref{Eq:g})
we derive that, as $c \to \infty$,
$g_c(x, r) \to g(x, r)$ uniformly in $(x, r)$ on every compact set in
$\{(x, r): x \in \R_+, r \in \R \backslash\{0, x\}\}$.
This implies that $g_c(x^{(c)}, r_c) \to g(x, r)$ whenever $r \ne x$ and
$r_c \to r$ as $c \to \infty$. Hence (A1)$'$ is satisfied.

To verify Condition (A2)$'$, we take a constant $\beta > \alpha$ such that
$\beta (H - 1/\alpha) > -1$ and consider the integral
\begin{equation}\label{Eq:g6}
\int_{|r|\le T} \bigg| \frac{1} {c^{H -1/\alpha}}\sum_{k =[cx]-[cr]+1}^\infty
\widetilde{c}_k \bigg|^\beta \,d \rho_c(r) = \sum_{|j| \le cT}\frac{1}
{c^{\beta(H -1/\alpha)+1}}\,
\bigg|\sum_{k =[cx]-j+1}^\infty
\widetilde{c}_k \bigg|^\beta
\end{equation}
Let $N > 1$ be a constant such that $|c_k| \le 2c_0 k^{H-1/\alpha -1}$ for all
$k \ge N$. We split the summation in the right-hand side of (\ref{Eq:g6})
according to whether $[cx]-j \le N$ or $[cx]-j > N$. Thanks to the fact that
$\sum_{k} \widetilde{c}_k= 0$ we have
\begin{equation}\label{Eq:g7}
\sum_{|j| \le cT,[cx]-j \le N }\frac{1}
{c^{\beta(H -1/\alpha)+1}}\,
\bigg|\sum_{k =[cx]-j+1}^\infty
\widetilde{c}_k \bigg|^\beta \le \frac{K_3}
{c^{\beta(H -1/\alpha)+1}}
\end{equation}
for some constant $K_3 > 0$. In the above we have also used the fact that there
are at most $N+1$ non-zero terms in the summation in $j$.

On the other hand, we have
\begin{equation}\label{Eq:g8}
\begin{split}
&\sum_{|j| \le cT,[cx]-j > N }\frac{1}
{c^{\beta(H -1/\alpha)+1}}\,
\bigg|\sum_{k =[cx]-j+1}^\infty
\widetilde{c}_k \bigg|^\beta \\
&\qquad \le   K_4
\sum_{|j| \le cT,[cx]-j > N } \frac{([x]-[j/c])^{\beta(H - 1/\alpha)}} c \,\\
&\qquad \le K_5 \int_{|r| \le T} |x-r|^{\beta(H - 1/\alpha)}\, dr
\end{split}
\end{equation}
for some constants $K_4,\, K_5 > 0$ and the last integral is convergent because
$\beta(H - 1/\alpha)>-1$.
Combining (\ref{Eq:g7}) and (\ref{Eq:g8}) yields
\begin{equation}\label{Eq:g9}
\begin{split}
&\int_{|r|\le T} \bigg| \frac{1} {c^{H -1/\alpha}}\sum_{k =[cx]-[cr]+1}^\infty
\widetilde{c}_k \bigg|^\beta \,d \rho_c(r) \\
&\le \frac{K_3}
{c^{\beta(H -1/\alpha)+1}} + K_5 \int_{|r| \le T} |x-r|^{\beta(H - 1/\alpha)}\, dr.
\end{split}
\end{equation}
Since the first sum in (\ref{Eq:g}) corresponds to $x = 0$, we see that
(A2)$'$ follows from (\ref{Eq:g9}).

The verification of (A3)$'$ is similar to the above and hence is omitted.
This finishes the proof of Theorem \ref{LRDthm3}.
\end{proof}

Finally we prove Theorem \ref{LRDthm}.

\begin{proof}[Proof of Theorem \ref{LRDthm}]
Recall that $\tilde b(c)^{-1}N_{ct} \Rightarrow E_t $ in the $J_1$ topology
\cite[Corollary 3.4]{limitCTRW}.  Theorem 4.6.1 in Whitt \cite{whitt} shows
that, as $n \to \infty$, $\sigma^{-1}_n S(nt) \Rightarrow W_H(t)$ in the $J_1$
topology on $D([0,\infty),\rr)$. This result is originally due
to Davydov \cite{davydov}, see also Giriatis et al. \cite[p.\ 276]{Giriatis03}.
Since the sequence $(J_n)$ is independent of $(Y_n)$, we have
$\left(\sigma^{-1}_{[\tilde b(c)]} S(\tilde b(c)t),\tilde b(c)^{-1}N_{ct}\right)
\Rightarrow (W_H(t),E_t)$ in the product
space, and then continuous mapping along with Theorem 13.3.1 in \cite{whitt}
yields \eqref{LRDconv} in the $J_1$ topology.
\end{proof}

\bigskip

\section{Discussion}

Self-similar processes arise naturally in limit theorems of random walks and
other stochastic processes, and they have been applied to model various phenomena
in a wide range of scientific areas including telecommunications, turbulence,
image processing and finance \cite{EmbrechtsMaejima}. The most prominent example
is fractional Brownian motion (FBM).  However, many real data sets are non-Gaussian,
which motivates the development of alternative models.
Many authors have constructed and investigated various classes of non-Gaussian
self-similar processes. Samorodnitsky and Taqqu \cite{ST94} provide a systematic
account on self-similar stable processes with stationary increments. Burdzy
\cite{burdzy1, burdzy2} introduced \emph{iterated Brownian motion} (IBM) which
replaces the time parameter of a two-sided Brownian motion by an independent
one-dimensional Brownian motion $B= \{B_t, t \ge 0\}$.
In this paper we have shown that the limit processes of CTRWs with correlated
jumps form a wide class of self-similar processes which are different from the
existing ones.

When $0<\beta\leq 1/2$, the inner process $E_t$ in \eqref{SRDconv} or
(\ref{LRDconv}) is also the local time at zero $L_t$ of a stable L\'evy
process, and the iterated process
$\{W_H(L_t), t \ge 0\}$ is called a local time fractional Brownian motion (LTFBM)
in \cite{MNX}, a self-similar process with index $\beta H$ and continuous sample
paths. Large deviation and modulus of continuity results for LTFBM are developed
in a companion paper \cite{MNX}. Strassen-type law of the iterated logarithm
has been proved by Cs\'{a}ki,  F\"{o}ldes and R\'{e}v\'{e}sz
\cite{CCFR2} for local time Brownian motion (LTBM, the case $H=1/2$).
It is interesting to note that our Theorem \ref{LRDthm} shows that
``randomly-stopped stationary sequence'' $\{(Y_n: n \le N_t), t \ge 0\}$
belongs to the domain of attraction of $\{W_H(L_t), t \ge 0\}$ for all
$H\in (0, 1)$. This theorem provides a physical interpretation of
the process $\{W_H(L_t), t \ge 0\}$.

One interesting property of LTBM is that its increments are
uncorrelated (this follows by simple conditioning argument), but not independent.
It has long been recognized that price returns are essentially uncorrelated, but
not independent \cite{BaillieReview, Mandelbrot}.  Hence LTBM, the scaling limit
of a CTRW with (weakly) correlated price jumps, may be useful to model financial
price returns. This approach could provide an interesting alternative to the
subordinated variance-Gamma model of Madan and Seneta \cite{VGmodel, CGMY} or
the FATGBM model of Heyde \cite{HeydeFATGBM}.

LTBM has a close connection to fractional partial
differential equations.  Meerschaert and Scheffler \cite{limitCTRW} and Baeumer
and Meerschaert \cite{fracCauchy} showed
that the probability density $u(x,t)$ of LTBM solves the fractional Cauchy problem.
\begin{equation}\label{frac-derivative-0}
\partial_t^\beta u(t,x) =
\partial_x^2 u(t,x) .
\end{equation}
Baeumer, Meerschaert and Nane \cite{bmn-07} further showed that the density of
the iterated Brownian motion solves the same equation \eqref{frac-derivative-0}.
As we mentioned at the end of Section 2, the connection between the limit
processes in this paper and fractional partial differential equations remains to
be investigated.

\end{document}